\documentclass[12pt,a4paper]{amsart}
\usepackage{amsfonts}
\usepackage{amsthm}
\usepackage{amsmath}
\usepackage{amscd}
\usepackage[latin2]{inputenc}
\usepackage{t1enc}
\usepackage[mathscr]{eucal}
\usepackage{indentfirst}
\usepackage{graphicx}
\usepackage{graphics}
\usepackage{pict2e}

\numberwithin{equation}{section}

\theoremstyle{plain}
\newtheorem{Th}{Theorem}
\newtheorem{Lemma}[Th]{Lemma}

 \theoremstyle{definition}

\newtheorem{?}[Th]{Problem}
\newcommand{\R}{\mathbb{R}}
\newcommand{\al}{\alpha}

\newcommand{\be}{\beta}

\newcommand{\Z}{\mathbb{Z}}

\begin{document}

\title{Minkowski's inequality and sums of squares}

\author[P. E. Frenkel]{P\'eter E. Frenkel}

\address{E\"{o}tv\"{o}s Lor\'{a}nd University \\ Department of Algebra and Number Theory
 \\ H-1117 Budapest
 \\ P\'{a}zm\'{a}ny P\'{e}ter s\'{e}t\'{a}ny 1/C \\ Hungary}
\email{frenkelp@cs.elte.hu}

\author[P. Horv\'ath]{P\'{e}ter Horv\'{a}th}

\email{horvath.peter17@gmail.com}
\thanks{}

\thanks{The first author's research is partially supported by MTA R\'enyi
"Lend\"ulet" Groups and Graphs Research Group.}



\begin{abstract} Positive polynomials arising from Muirhead's inequality,  from classical power mean and elementary symmetric mean  inequalities and from Minkowski's inequality can be rewritten as sums of squares.
\end{abstract}

\maketitle

\section{Introduction}
Many of the most important inequalities in mathematics are, or can be reformulated as, algebraic inequalities. An algebraic inequality is one that asserts that some given polynomial is nonnegative everywhere (or nonnegative on some specified set).

A  polynomial $f\in\R[x_1,\dots,x_n]$ that is nonnegative everywhere is not necessarily a  sum of  squares (of polynomials).  This fact was conjectured by Minkowski and proved by Hilbert. The simplest known counterexample has been given by Motzkin. However, if the inequality $f\ge 0$ is `classical' and `famous' enough, then $f$ usually turns out to be representable as a sum of squares, although such a  representation is not always easy to find.  For example, the most standard proof of the Cauchy--Schwarz inequality is \it not \rm the one that rewrites the difference of the two sides as a sum of squares, but such a  rewriting  is  possible (and almost as well known). More interestingly, the inequality between the arithmetic and the geometric mean also has such a proof, as was demostrated by Hurwitz \cite{H} in 1891.  The paper of Fujisawa \cite{F} gives numerous further examples of this phenomenon. 

Such a  purely algebraic proof of an algebraic inequality, even if it is not the simplest proof,  gives some extra understanding of why the inequality `must be true'. 

 In the present note, we  give  square sum decompositions of positive polynomials arising from the   inequalities listed below. In each of these, the variables $x_i$, $X_i$, $Y_i$ are meant to be \emph{nonnegative} reals. 
\begin{itemize}
\item The inequality
\begin{equation}\label{power}\sqrt[q]{\frac1n\sum_{i=1}^{n}x_i^q} \leq\sqrt[p]{\frac1n\sum_{i=1}^{n}x_i^p} \end{equation} between power means. This holds for any real exponents  $p\ge q>0$, and can be rewritten as an algebraic inequality when $p$ and $q$ are integers.

\item The more general inequality of Lyapunov:
\begin{equation}\label{Lyap}\left({
\sum_{i=1}^{n}x_i^q}\right)^{p-r} \leq\left({
\sum_{i=1}^{n}x_i^p}\right)^{q-r}\left({
\sum_{i=1}^{n}x_i^r}\right)^{p-q} .\end{equation}  This holds for any real exponents $p\ge q\ge r\ge  0$, and is an algebraic inequality when $p$, $q$ and $r$ are integers. Note that  the special case  $r=0$ is the preceding power mean inequality.

\item Maclaurin's inequality
\begin{equation}\label{Macl}
\sqrt[q]
{\binom{n}{q}^{-1}\sum_{1\le i_1<\dots <i_q\le n} x_{i_1}\cdots x_{i_q}}
 \geq
\sqrt[p]{\binom{n}{p}^{-1}\sum_{1\le i_1<\dots <i_p\le n}x_{i_1}\cdots x_{i_p}} \end{equation} between elementary symmetric means.  This holds for integers  $n\ge p\ge q\ge 1$, and can be rewritten as an algebraic inequality. Note that the special case  $q=1$, $p=n$ is the inequality between the arithmetic and the geometric mean.

\item A Lyapunov type generalization of Maclaurin's inequality:
\begin{equation}\label{Maclyapunov}
\left(
\frac{\sum_{ i_1<\dots <i_q} x_{i_1}\cdots x_{i_q}}{\binom nq}\right)^{p-r} 
 \geq
\left(\frac{\sum_{ i_1<\dots <i_p}x_{i_1}\cdots x_{i_p}}{\binom np}\right)^{q-r} \left(\frac{\sum_{ i_1<\dots <i_r}x_{i_1}\cdots x_{i_r}}{\binom nr}\right)^{p-q}\end{equation}
 for the elementary symmetric means.  This holds for integers  $n\ge p\ge q\ge r\ge 0$, and is an algebraic inequality.  Note that the special case  $r=0$ is Maclaurin's inequality, and the special case $r=q-1$,  $p=q+1$ is Newton's well-known inequality. This latter special case is easily seen to imply the general case.

\item 
Minkowski's inequality (superadditivity of the geometric mean):
\begin{equation}\label{Mink}\sqrt[n]{\prod_{i=1}^{n}X_i} + \sqrt[n]{\prod_{i=1}^{n}Y_i} \leq \sqrt[n]{\prod_{i=1}^n(X_i+Y_i)}.\end{equation}

\end{itemize}

The treatment of all of these will rely on rewriting \it Muirhead's inequality \rm as a sum of squares (see  Lemma~\ref{Muir}). This inequality is more technical and so we postpone it to Section~\ref{Means}.

\section{Nonnegative variables}
Classical inequalities often involve \it nonnegative \rm real variables as opposed to real variables. Note that all of our examples above have been stated for nonnegative variables, although in some special cases the nonnegativity assumption can be dropped.

In the setting of nonnegative variables, the suitable analog of  the semiring of sums of squares is  the semiring \[S=S_n=\left\{\sum_{\varepsilon_1=0}^{1} \ldots \sum_{\varepsilon_n=0}^{1}r_{\underline{\varepsilon}} \prod_{j=1}^nx_j^{\varepsilon_j}  \;|\;r_{\underline{\varepsilon}}\textrm{ is a sum of squares in }\R[x_1,\dots,x_n]\right\}.
\]  It is immediately seen that $S$ is indeed a  semiring, i.e., it is closed under addition and multiplication.  In fact, $S$ is the semiring generated by the variables $x_1$, \dots, $x_n$ and by the squares of all polynomials. 

Note that $p\in S$ implies that $p$ is nonnegative for $x_1,\dots, x_n\ge 0$, but not conversely.
Clearly, $p$ is nonnegative for $x_1,\dots, x_n\ge 0$ if and only if  $ p(x_1^2, \ldots, x_n^2)$ is nonnegative everywhere. The relevance of the semiring $S$ is explained by the following Lemma.

\begin{Lemma}
Let $p \in \mathbb{R}[x_1, \ldots, x_n]$. Then
$p \in S $ if and only if $ p(x_1^2, \ldots, x_n^2)$ is a sum of squares in $\mathbb{R}[x_1, \ldots, x_n]$.
\end{Lemma}

To appreciate the results proved in later sections of this paper, the trivial `only if' statement will be important. We included the more difficult
`if' statement to make the picture complete.

\begin{proof}
The `only if' part is trivial.
For the `if' part, we consider the  linear operator \rm
\begin{align*}\mathcal{R} : \mathbb{R}[x_1, \ldots, x_n] &\rightarrow \mathbb{R}[x_1^2, \ldots, x_n^2]
\end{align*}
 that maps the monomial $\prod x_j^{k_j}$ to itself if all $k_j$ are even and maps it to zero otherwise.
We assume that  $p(x_1^2, \ldots, x_n^2)$ is  a sum of squares, i.e., there exist polynomials $r_i$ such that
$$q(x_1, \ldots, x_n):=p(x_1^2, \ldots, x_n^2) = \sum_{i=1}^{k} r_i^2(x_1, \ldots, x_n).$$
In $r_i$, we group terms according to the parity of the exponents of $x_1$, \dots, $x_n$. We define the polynomials $r_{i,\underline{\varepsilon}}
$ for each $\underline{\varepsilon} = (\varepsilon_1, \ldots, \varepsilon_n)\in\{0,1\}^n$ so that
$$r_i(x_1, \ldots, x_n) = \sum_{\varepsilon_1=0}^{1} \ldots \sum_{\varepsilon_n=0}^{1}  r_{i,\underline{\varepsilon}} (x_1^2, \ldots, x_n^2)\prod_{j=1}^n x_j^{\varepsilon_j}. $$
Apply  $\mathcal{R}$ to $r_i^2$, then
$$(\mathcal{R}r_i^2)(x_1, \ldots, x_n) = \sum_{\varepsilon_1=0}^{1} \ldots \sum_{\varepsilon_n=0}^{1}  r_{i,\underline{\varepsilon}}^2 (x_1^2, \ldots, x_n^2)\prod_{j=1}^n x_j^{2\varepsilon_j}.$$
Hence,
$$p(x_1^2, \ldots, x_n^2) = q(x_1, \ldots, x_n) = (\mathcal{R}q)(x_1, \ldots, x_n) =$$ $$=\sum_{i=1}^{k} (\mathcal{R}r_i^2)(x_1, \ldots, x_n) =  \sum_{i=1}^{k} \sum_{\varepsilon_1=0}^{1} \ldots \sum_{\varepsilon_n=0}^{1} r_{i,\underline{\varepsilon}}^2 (x_1^2, \ldots, x_n^2)\prod_{j=1}^n  x_j^{2\varepsilon_j}. $$
\\
Therefore, \[\displaystyle p(x_1, \ldots, x_n) = \sum_{i=1}^{k} \sum_{\varepsilon_1=0}^{1} \ldots \sum_{\varepsilon_n=0}^{1}  r_{i,\underline{\varepsilon}}^2 (x_1, \ldots, x_n)\prod_{j=1}^n x_j^{\varepsilon_j},\] whence $p \in S$.
\end{proof}

\bigskip

\section{Means}\label{Means}
We now wish to generalize a few results of Fujisawa \cite{F} concerning power mean and elementary symmetric mean inequalities.

Fix a  nonnegative integer $d$, the degree of the homogeneous polynomials we will be looking at.
Let us consider the set of \it   partitions \rm  of  $d$ into at most $n$ parts (such a partition is a  weakly decreasing $n$-term sequence of nonnegative integers adding up to $d$). There is a standard partial order on this set. First of all, we write $\alpha\searrow\beta$ if, for some indices $k<l$, we have $\beta_k=\alpha_k-1$, $\beta_l=\alpha_l+1$  and $\beta_i=\alpha_i$ for $i\ne k,l$. Then, we define the partial order $\succeq$ to be the reflexive transitive closure of $\searrow$.  I.e., $\alpha\succeq\beta$ if and only if there exists an $N\ge 0$ and a sequence of partitions $\alpha=\alpha_0\searrow\alpha_1\searrow\dots\searrow \alpha_N=\beta$.  We write $\alpha\succ\beta$ if $\alpha\succeq\beta$ and $\alpha\ne\beta$. We mention, but will not make use of, the well-known fact that $\alpha\succeq\beta$
holds if and only if $\alpha_1+\dots +\alpha_k\ge \beta_1+\dots +\beta_k$ for all $k$.

We now introduce  the \it Reynolds operator \rm $\mathcal R$ of the symmetric group $\mathfrak S_n$. For  a polynomial $f\in \R[x_1,\dots, x_n]$, let \[(\mathcal Rf)(x_1,\dots, x_n):=\displaystyle \frac 1{n!}\sum_{\sigma \in \mathfrak S_n} f(x_{\sigma (1)}, \ldots, x_{\sigma (n)}).\]  We introduce the monomial $x^\alpha=\prod x_j^{\alpha_j}$ and define the normalized monomial symmetric function $[\alpha]=\mathcal Rx^\alpha$.

We have

\bigskip

\bf Muirhead's inequality. \rm  If the  partitions $\al$ and $\be$ satisfy $\alpha\succeq\beta$, then 
  $[\alpha](x_1,\dots,x_n)\ge  [\beta] (x_1,\dots, x_n) $ for all nonnegative $x_1$, \dots, $x_n$.

\bigskip

The following lemma will be crucial for the sequel.

\begin{Lemma}[Muirhead's inequality rewritten as sums of squares]\label{Muir} If the  partitions $\al$ and $\be$ satisfy $\alpha\succeq\beta$, then 
  $[\alpha] - [\beta] \in S$.
\end{Lemma}

\begin{proof} Since $\succeq$ is the reflexive transitive closure of $\searrow$, and $S$ is closed under addition, we may assume that $\alpha\searrow\beta$.  Let $k<l$ be as in the definition of $\searrow$.
Then $
[\alpha] - [\beta]=$  \begin{align*}&= \frac{1}{2 } \mathcal R\left(  \left( x_k^{\al_k}x_l^{\al_l} + x_k^{\al_l}x_l^{\al_k} -
x_k^{\al_k-1}  x_l^{\al_l+1} -
x_k^{\al_l+1}  x_l^{\al_k-1} \right)\prod_{i\ne k,l}  x_i^{\alpha_i}\right) .
\end{align*} The four-term expression in the inner parantheses equals
$$ = (x_k-x_l)^2\left(x_k^{\al_k-2}x_l^{\al_l}+x_k^{\al_k-3}x_l^{\al_l+1} + \ldots +
x_k^{\al_l+1}x_l^{\al_k-3} + x_k^{\al_l}x_l^{\al_k-2}\right).$$ This is in $S$, therefore so is the whole expression. 
\end{proof}

We now turn to \it  normalized power sums \rm
\[P_k=[k,0,\dots, 0]=\frac1n\sum_{i=1}^nx_i^k\] and \it normalized elementary symmetric polynomials \rm
\[E_k=[\underbrace{1,\ldots, 1}_k, 0, \ldots, 0]=\binom nk^{-1}\sum_{i_1<\dots< i_k}x_{i_1}\cdots x_{i_k}.\] 
Note that $P_0=E_0=1$.
\begin{Lemma}
For $1\le i\le k$, we have \[P_{i-1}P_{k+1}-P_{i}P_{k}\in S\qquad \text{ and }\qquad 
E_{i}E_{k}-E_{i-1}E_{k+1}\in S.\]
\end{Lemma}

For $i=1$ and for $i=k$, this was shown by Fujisawa in \cite{F}. 

\begin{proof}
We have
\begin{align*}
P_{i-1}P_{k+1}-P_{i}P_{k} =\\= [i-1,0,\ldots,0] \cdot [k+1,0,\ldots,0]-[i,0,\ldots,0] \cdot [k,0,\ldots,0] = \\
= \frac{1}{n} [i+k,0,\ldots,0]+\frac{n-1}{n}[k+1,i-1,0,\ldots,0]+\\-
\frac{1}{n} [i+k,0,\ldots,0] - \frac{n-1}{n} [k,i,0,\ldots,0] = \\
=\frac{n-1}{n} \left( [k+1,i-1,0,\ldots,0] - [k,i,0,\ldots,0] \right).
\end{align*}
Since $(k,i,0,\ldots,0) \prec (k+1,i-1,0,\ldots,0)$, the first statement follows by the previous lemma. Explicitly, we get
\[ P_{i-1}P_{k+1}-P_{i}P_{k}=\frac{n-1}{2n } 
 \mathcal R\left( (x_1-x_2)^2 \sum_{j=i-1}^{k-1} x_1^j x_2^{i+k-2-j} \right)
.\]

For the second statement, put
$\be_{(r)}=(\underbrace{2,\ldots,2}_{r},\underbrace{1,\ldots,1}_{i+k-2r},\underbrace{0,\ldots,0)}_{n-i-k+r}$. Then
$$E_i \cdot E_{k}={\binom{n}{i}^{-1}} \sum_{r=0}^{i} \binom{k}{r} \binom{n-k}{i-r}[\be_{(r)}].$$
For all  $0\leq r \leq i-1$, we have  $\be_{(r)} \prec \be_{(r+1)}$, so $[\be_{(r+1)}]-[\be_{(r)}]$ is in $S$ by the previous lemma. It will suffice to find nonnegative constants $a_r$ such that
$$E_{i}E_{k}-E_{i-1}E_{k+1}=\sum_{r=0}^{i-1} a_r \left([\be_{(r+1)}]-[\be_{(r)}]\right).$$  It is easy to see without calculation that such nonnegative constants exist, but we still do the calculation in order to get an explicit formula.

Put $a_k=a_{-1}=0$, then the coefficient of $[\be_{(r)}]$ will be $a_{r-1}-a_{r}$ for all $r$. Therefore what we need to achieve is
$a_{r-1}-a_r =$ 
\begin{align*}
&= {\binom{n}{i}^{-1}}  \binom{k}{r} \binom{n-k}{i-r}-{\binom{n}{i-1}^{-1}}  \binom{k+1}{r} \binom{n-1-k}{i-1-r}
 = \\
&=  {\binom{n}{i}^{-1}}  \binom{k}{r} \binom{n-1-k}{i-1-r}\left(\frac{n-k}{i-r}-\frac{n+1-i}i\frac{k+1}{k+1-r}\right).
\end{align*}
Examining the special cases $r=0$ and $r=1$, we 
 are led to conjecture
$$a_r= {\binom{n}{i}^{-1}}  \binom{k}{r} \binom{n-1-k}{i-1-r}\frac{k+1-i}i.$$
An easy calculation shows that this indeed gives the correct $a_{r-1}-a_{r}$.
The statement follows; explicitly,
$$E_{i}E_{k}-E_{i-1}E_{k+1}=  \frac{1}{2 } \sum_{r=0}^{i-1} a_r \mathcal R\left( 
(x_1-x_2)^2 x_3^2 \ldots x_{r+2}^2 x_{r+3} \ldots x_{2k-r}\right).$$
\end{proof}

\begin{Th}[Power mean inequality and Maclaurin's inequality rewritten as sums of squares] For $p\ge q$, we have $P_p^q-P_q^p\in S$ and $E_q^p-E_p^q\in S$. 
\end{Th}
\begin{proof} We have
$$P_p^q-P_q^p=P_p^q  P_0^{p-q}- P_q^p=$$ $$= \sum_{i=1}^{q} \sum_{k=q}^{p-1} (P_{i-1}P_{k+1}-P_{i}P_{k})
P_{i-1}^{k-q} P_{i}^{p-(k+1)}  P_{p}^{q-i} P_{q}^{i-1},$$
because the latter expression equals
\begin{eqnarray*}
\sum_{i=1}^{q} P_{p}^{q-i} P_{q}^{i-1} \sum_{k=q}^{p-1} \Big( P_{i-1}^{k+1-q} P_{i}^{p-(k+1)} P_{k+1} -
 P_{i-1}^{k-q} P_{i}^{p-k} P_k \Big) &=&\\
= \sum_{i=1}^{q} P_{p}^{q-i} P_{q}^{i-1} \Big( P_{i-1}^{p-q}P_p - P_i^{p-q}P_q  \Big)   &=&\\
= \sum_{i=1}^{q}\left( P_{i-1}^{p-q} P_p^{q-(i-1)} P_q^{i-1} - P_i^{p-q} P_p^{q-i} P_q^i  \right) &=&\\
 = P_p^q  P_0^{p-q}- P_q^p.
\end{eqnarray*}
\\
We have thus represented $P_p^q-P_q^p$ as an $S$-linear combination of the polynomials $P_{i-1}P_{k+1}-P_{i}P_{k}$  $(1 \leq i \leq q \leq k \leq p-1)$. The first claim now follows from that of the previous lemma.  We omit the proof of the second statement because is it essentially the same.
\end{proof}

\begin{Th}[Lyapunov's inequality and the generalized  Maclaurin inequality rewritten as sums of squares] For $p\ge q\ge r$, we have $P_p^{q-r}P_r^{p-q}-P_q^{p-r}\in S$ and $E_q^{p-r}-E_p^{q-r}E_r^{p-q}\in S$. 
\end{Th}

\begin{proof} Only a slight modification of the previous proof is needed. We have
$$P_p^{q-r}P_r^{p-q}-P_q^{p-r}= \sum_{i=r+1}^{q} \sum_{k=q}^{p-1} (P_{i-1}P_{k+1}-P_{i}P_{k})
P_{i-1}^{k-q} P_{i}^{p-(k+1)}  P_{p}^{q-i} P_{q}^{i-1-r},$$
because the latter expression equals
\begin{eqnarray*}
 \sum_{i=r+1}^{q} P_{p}^{q-i} P_{q}^{i-1-r} \sum_{k=q}^{p-1} \Big( P_{i-1}^{k+1-q} P_{i}^{p-(k+1)} P_{k+1} -
 P_{i-1}^{k-q} P_{i}^{p-k} P_k \Big) &=&\\
= \sum_{i=r+1}^{q} P_{p}^{q-i} P_{q}^{i-1-r} \Big( P_{i-1}^{p-q}P_p - P_i^{p-q}P_q  \Big)   &=&\\
= \sum_{i=r+1}^{q} \left(P_{i-1}^{p-q} P_p^{q-(i-1)} P_q^{i-1-r} - P_i^{p-q} P_p^{q-i} P_q^{i-r}\right)   &=&\\
 = P_p^{q-r}  P_r^{p-q}- P_q^{p-r}.
\end{eqnarray*}
\\
The first claim  follows.  The proof of the second statement  is  essentially the same.
\end{proof}

\section{Minkowski's inequality}
We now come to our main result, which concerns
Minkowski's inequality~(\ref{Mink}).
In this case, we cannot get a polynomial inequality by simply raising both sides to some power. We also need to  substitute $X_i=x_i^n$ and  $Y_i=y_i^n$ to get the $2n$-variable polynomial inequality
 $$P(x_1, \ldots, x_n, y_1, \ldots, y_n) := \prod_{i=1}^n(x_i^n+y_i^n) - \left(\prod_{i=1}^{n}x_i + \prod_{i=1}^{n}y_i\right)^n\ge 0$$ for $x_i,y_i\ge 0$.

\begin{Th}[Minkowski's inequality rewritten as sum of squares] We have $P\in S$ (where $S=S_{2n}$ since we are dealing with $2n$ variables).
\end{Th}
This answers a question raised by Andr\'es Caicedo on his teaching blog \cite{C}.

\begin{proof}
We expand $P$ and group terms according to the number of factors $x$ they contain. This gives $$P
 =\sum_{k=0}^{n} \left( \sum_{\substack{I \subseteq \{1,\ldots,n \} \\ \mid I \mid = k }} \prod_{i\in I}x_i^n \prod_{j\notin I}y_j^n -  \binom{n}{k} \prod_{i=1}^{n} x_i^{k}y_i^{n-k} \right).$$
To facilitate notation, we think of the index set $\{1,\dots,n\}$ as $\Z/n\Z$, the integers modulo $n$.
It will suffice to prove that for all   $k\in \{0, \ldots, n\}$, the polynomial in  large parantheses
 is in $S$. That polynomial equals
\begin{eqnarray*}
 \sum_{\substack{I \subseteq \Z/n \Z \\ \mid I \mid = k }} \left( \frac{1}{n} \sum_{t=0}^{n-1}   \prod_{i \in I+t}x_i^n
\prod_{j \notin I+t}y_j^n  -\prod_{i=1}^{n}x_i^k y_i^{n-k} \right) = \\
=
\sum_{\substack{I \subseteq \Z/n \Z \\ \mid I \mid = k }} \left(  \frac{1}{n}\sum_{t=0}^{n-1}  \left( \prod_{i \in I+t}x_i
\prod_{j \notin I+t}y_j\right) ^n - \prod_{t=0}^{n-1} \left( \prod_{i \in I+t}x_i
\prod_{j \notin I+t}y_j \right) \right). \\
\end{eqnarray*}
For a  fixed set $I$ of indices, denote $z_t=\displaystyle \prod_{i \in I+t}x_i \prod_{j \notin I+t}y_j$. Then  $z_t \in S_{2n}$.  Note that   $S_n$   contains the polynomial \[f(x_1,\dots, x_n)=\frac{x_1^n+\ldots + x_n^n}{n} - x_1 \cdots x_n=[n,0,\dots,0]-[1,1,\dots,1]\] by Lemma~\ref{Muir}. Therefore, $S_{2n}$   contains the polynomial 
$$ f(z_0,\dots, z_{n-1})=\frac{1}{n}\sum_{t=0}^{n-1} z_t^n - \prod_{t=0}^{n-1} z_t,$$ whence $P \in S$.
\end{proof}

\end{document}